\documentclass[review]{elsarticle}

\usepackage[T1]{fontenc}
\usepackage{lmodern}
\usepackage{geometry}

\usepackage{amsthm,amsfonts}
\usepackage[sumlimits]{amsmath}
\usepackage{amssymb}
\usepackage[all]{xy}
\usepackage{tikz}
\usepackage{graphicx}
\usepackage{url}
\usepackage{tikz-cd}
\usepackage{pgf}
\usepackage{ragged2e}
\usepackage{hyperref}

\usepackage[utf8]{inputenc} 
\usepackage[T1]{fontenc}    
\usepackage{booktabs}       
\usepackage{nicefrac}       
\usepackage{microtype}      
\usepackage{lipsum}

\newtheorem{lema}{Lemma}[section]
\newtheorem{teo}[lema]{Theorem}

\theoremstyle{definition}

\numberwithin{equation}{section}

\newcommand{\eg}{\emph{e.g.} }
\newcommand{\cf}{\emph{cf.} }
\newcommand{\tq}{\,|\,}
\newcommand{\ie}{\emph{i.e.}}

\newcommand{\apres}[2]{\langle #1 \tq #2\rangle}

\newcommand{\ed}{\ar@{-}}

\renewcommand{\epsilon}{\varepsilon}

\numberwithin{equation}{section}

\newcommand{\BZ}{\mathbb{Z}}

\begin{document}
\begin{frontmatter}

\title{Subgroup Separability of Artin Groups}

\author[1]{Kisnney Almeida\corref{cor1}}
\ead{kisnney@gmail.com}

\author[2]{Igor Lima}
\ead{igor.matematico@gmail.com}

\address[1]{Departamento de Ci\^encias Exatas,
Universidade Estadual de Feira de Santana, 44036-900, Feira de Santana-BA,
Brazil}

\address[2]{Departamento de Matem\'atica,
Universidade de Bras\'ilia,
70910-900, Bras\'ilia - DF,
Brazil}

\cortext[cor1]{Corresponding author}




\begin{abstract}
We find a condition on the underlying graph of an Artin group that fully determines if it is subgroup separable. As a consequence, an Artin group is subgroup separable if and only if it can be obtained from the irreducible spherical Artin groups of ranks 1 and 2 via a finite sequence of free products and direct products with the infinite cyclic group. This result generalizes the Metaftsis-Raptis criterion for Right-Angled Artin groups.
\end{abstract}



\end{frontmatter}

\section{Introduction}

In this paper we investigate subgroup separability of Artin groups.

We begin by recalling that the cosets to the finite index normal subgroups of $G$ form a basis for the profinite topology of $G$. A group $G$  is called {\bfseries subgroup separable} or {\bfseries locally extended residually finite} or just {\bfseries LERF } if every finitely generated subgroup $H \leq G$ is closed in the profinite topology of $G$, or equivalently if every finitely generated subgroup of $G$ is equal to a intersection of subgroups of finite index of $G$.

Subgroup separability is a powerful property introduced by Hall \cite{H}, important for group theory and low-dimensional topology, but established either positively or negatively for very few classes of groups. It can be used to show that certain immersions lift to embeddings in a finite cover (\cite{S}, \cite{S2}, \cite{W}, \cite{LR}, \cite{Wi}). It is also known that, for finitely presented groups, Subgroup Separability implies the generalized word problem is solvable, \ie, there is an algorithm that decides for every element $g \in G$ and every finitely generated subgroup $H \leq G$ whether $g$ is in $H$ or not \cite{M}. Subgroup Separability is also related to other properties of groups (\cf \cite{Mi}, for example).

An Artin group is a finitely presented group defined from a finite simplicial labeled graph. Artin groups form a huge class of groups, which includes free groups, free abelian groups, the classical Braid groups and many other subclasses.  There are a lot of conjectures on Artin groups but most of them with no complete answer: for example, it is conjectured (and still open) that all Artin Groups are torsion-free, have solvable word problem, are centerless (except for the spherical ones) etc. These and many other facts and unanswered questions may be found at \cite{Ca}.

Metaftsis and Raptis \cite{MR} have obtained a complete criterion for the subgroup separability of an important subclass, the Right-Angled Artin Groups (RAAGs). The case of Braid groups was also solved by Dasbach and Mangum \cite{DM} but so far no general results on other subclasses of Artin groups have been explicitly obtained.

Jankiewicz and Schreve \cite{JS} have recently stated the Generalized Tits Conjecture, which presents for each Artin group a subgroup that is also a RAAG. They also prove the conjecture holds for some subclasses of Artin groups. That allowed us to generalize Metaftsis-Raptis criterion and establish a pattern for subgroup separable Artin groups.

From now on, all our graphs are finite, simplicial and have their edges labeled by integers greater than 1. We say two vertices are {\bfseries 2-adjacent} if they are adjacent by an edge labeled by 2. We define a {\bfseries 2-cone} of a graph $\Gamma'$ (with apex $u$) as the graph $\Gamma$ with $V(\Gamma)=V(\Gamma')\sqcup \{u\}$ and such that $u$ is 2-adjacent in $\Gamma$ to every vertex of $\Gamma'$. Let $\mathcal{S}$ be the class of graphs  obtained from graphs with at most two vertices by taking disjoint unions and 2-cones. Then our main result is

 \medskip
{\bf Theorem A.}
  {\it {Let $A=A(\Gamma)$ be an Artin group, with $\Gamma$ as its underlying graph. Then $A$ is subgroup separable if and only if $\Gamma\in \mathcal{S}$.}}\\

Let $A(\mathcal{S})$ be the smallest class of Artin groups such that:
\begin{enumerate}
  \item
  $A(\mathcal{S})$ contains all Artin groups of ranks at most 2;

  \item
  If $A, B\in A(\mathcal{S})$ then the free product $A*B\in A(\mathcal{S})$;

  \item
  If $A\in A(\mathcal{S})$ then $A\times \BZ\in A(\mathcal{S})$.
\end{enumerate}

We will see that $A(\mathcal{S})$ coincides with the class of Artin groups whose underlying graphs belong to $\mathcal{S}$, as expected from the notation.  So the following result is an easy consequence of Theorem A.

 \medskip
{\bf Corollary A.}
  {\it {An Artin group is subgroup separable if and only if it belongs to $A(\mathcal{S})$.}}\\

The paper is organized as follows: In section 2 we recall some preliminaries on Artin groups; in section 3 we establish some concepts and results on subgroup separability that we will need; in section 4 we prove the new results.

\section{Artin Groups}

Let $\Gamma$ be a finite simplicial graph, with edges labeled by integers greater than 1. Then the {\bfseries Artin group with  $\Gamma$ as underlying graph}, denoted by $A(\Gamma)$, is given by a finite presentation, with generators corresponding to the vertices of $\Gamma$ and relations given by

  \[  \underbrace{abab\cdots}_{m \text{ factors}}=\underbrace{baba\cdots}_{m \text{ factors}} \]

    for each edge of $\Gamma$ that connects the vertices $a$ and $b$ labeled by $m$. In that case we will say  $a$ and $b$ are {\bfseries $m$-adjacent}.

Let $V(\Gamma)$ be the set of vertices of $\Gamma$ and $E(\Gamma)$ be its set of edges. A subgraph $\gamma$ of $\Gamma$ is called a {\bfseries full} subgraph (generated by $V(\gamma)$) if every two vertices of $\gamma$ that are adjacent in $\Gamma$ are also adjacent in $\gamma$. We will say $\gamma$ is a {\bfseries full path} of $\Gamma$ if it is a path and a full subgraph of $\Gamma$.

It is known that a subset $V_0\subset V(\Gamma)$ generates a subgroup of $A(\Gamma)$ that is isomorphic to $A(\Gamma_0$), where $\Gamma_0$ is the full subgraph of $\Gamma$ generated by $V_0$ \cite{L} - such subgroups are called {\bfseries special subgroups} of $A(\Gamma)$.

Although the definition above is very simple, it is usually very hard to prove general results on Artin groups. That gave rise to relevant subclasses, some of which we present below.

An Artin group is called {\bfseries Right-Angled}, or just {\bfseries RAAG}, if all the labels of the edges of $\Gamma$ are equal to 2. If there is no ambiguity, for a RAAG it is usual to assume $\Gamma$ has no labels for simplicity of notation.

 An Artin group $A(\Gamma)$ is called {\bfseries of finite type} or {\bfseries spherical} if the quotient of $A(\Gamma)$ by the normal closure of $\{a^2 \tq a\in V(\Gamma)\}$ (which is always a Coxeter group) is finite. Just like Coxeter groups, a finite type Artin group is a direct product of {\bfseries irreducible} finite type Artin groups. The latter groups are divided into types, whose complete list we give below:

\begin{equation*}
 A_n (n\geq 1),\quad B_n (n\geq 2),\quad D_n (n\geq 4), \quad E_6,\quad E_7,\quad E_8,\quad F_4,\quad H_3,\quad H_4,\quad I_2(p)\, (p\geq 5).
 \end{equation*}

The detailed definitions for each of these groups may be found at \cite{P}, for example. It is worth mentioning that their underlying graphs are all complete (which means every two vertices are adjacent). The {\bfseries rank} of an irreducible spherical Artin group is the number of vertices in its underlying graph and may be identified by its type subscript. The irreducible spherical Artin groups are much more understood than most Artin groups. A well-known fact we need about these groups is that they have infinite cyclic center \cite{D}.

A remark on notations: as explained in the first pages of \cite{Ca}, there is some variability on the notations used for these groups and their underlying graphs. Under McCammond's classification, we use the modern terminology, although some of our references use the classical one, especially when mentioning the finite type Artin groups.

Finally, an Artin group $A(\Gamma)$ is called {\bfseries locally reducible} if all its spherical special subgroups are direct products of irreducible spherical Artin groups of rank 1 or 2. This is equivalent to the following condition: If $\Gamma$ contains a triangle whose edges are labeled $(2,m,n)$, with $m,n\geq 3$, then $\frac{1}{m}+\frac{1}{n}\leq \frac{1}{2}$ \cite{C}.

One of the few theorems that are known to hold for all Artin groups is the former called Tits Conjecture, now fully proved by Crisp and Paris after a series of partial results:

\begin{teo}[\cite{CP}]\label{teocp}
  For every $N\geq 2$ the subgroup of $A(\Gamma)$ generated by $\{v^N \tq v\in V(\Gamma)\}$ is a RAAG $A(\Gamma')$ such that $V(\Gamma')=V(\Gamma)$ and $v,w$ are adjacent in $\Gamma'$ if and only if $v,w$ are 2-adjacent in $\Gamma$.
\end{teo}

Recently  Jankiewicz and Schreve \cite{JS} have stated a generalized version of Tits Conjecture, which we explain below.

If $A=A(\Gamma)$ is an Artin group,  let $\mathbb{S}=\mathbb{S}(A)$ be the set of all full subgraphs $\Sigma$ of $V(\Gamma)$ such that $A(\Sigma)$ is an irreducible spherical Artin group. If $\Sigma, \Lambda\in \mathbb{S}$ we say $[\Sigma,\Lambda]=1$ if every two distinct vertices $a\in V(\Sigma)$ and $b\in V(\Lambda)$  are 2-adjacent in $\Gamma$. Let $R(A)$ be the RAAG generated by the set $\{z_{\Sigma}\}_{\Sigma\in \mathbb{S}}$ with presentation
$$R(A)=\apres{z_{\Sigma}}{ [z_{\Sigma},z_{\Lambda}]=1 \text{ if  $\Sigma\subset \Lambda$, $\Lambda\subset \Sigma$ or $[\Sigma, \Lambda]=1$}}.$$

{\bfseries Remark:} To ease notation, we slightly changed the definition of $R(A)$; we are identifying any set of vertices with the full subgraph generated by them. It's not a problem because there is a natural bijection between those two sets. We will do the same in the proofs of section 4.

Jankiewicz and Schreve define a homomorphism $\phi_N:R(A)\to A$ for each natural $N$ and conjecture that for each Artin group there is a $N$ such that $\phi_N$ is injective. We don't need all the details, so we will use a short version of the conjecture:

 \medskip
{\bf Generalized Tits Conjecture (short version)}
  {\it Let $A=A(\Gamma)$ be an Artin group. Then there is a subgroup of $A$ that is isomorphic to the group $R(A)$ defined above.}\\

The original conjecture obviously implies this short version. The authors have proven the conjecture holds for a lot of Artin groups, which we list below:

\begin{teo}[\cite{JS}]\label{teojs}
The Generalized Tits Conjecture holds for Artin groups $A(\Gamma)$ such that
\begin{enumerate}
  \item
  $\Gamma$ has no edges labeled by 3;

  \item
  $A(\Gamma)$ is locally reducible;

  \item
  $A(\Gamma)$ is irreducible spherical and not of type $E_6$, $E_7$, $E_8$.
\end{enumerate}
\end{teo}

\section{Subgroup Separability}

First we recall the behaviour of subgroup separability on subgroups.

\begin{teo}[\cite{S},\cite{S2}]\label{teosublerf}
Let $H$ be a subgroup of $G$. If $H$ is not LERF then $G$ is not LERF either. The converse holds if $[G:H]<\infty$.
\end{teo}

We strongly use the Metaftsis-Raptis criterion on subgroup separability for Right-Angled Artin groups.

\begin{teo}[\cite{MR}]\label{teoraag}
  A Right-Angled Artin Group $A(\Gamma)$ is LERF if and only if $\Gamma$ does not contain a path of length three nor a square as full subgraphs.
\end{teo}

The subgraphs mentioned in Theorem \ref{teoraag} correspond to two non-LERF possible subgroups of $A(\Gamma)$: the first is the fundamental group of the complement of
the chain of four circles \cite{NW} (known as $L$) and the second is isomorphic to $F_2\times F_2$, where $F_2$ is the free group of rank 2.

We will need some other results on subgroup separability:

\begin{teo}[\cite{B}]\label{teofreeprod}
  A free product of LERF groups is LERF.
\end{teo}

\begin{teo}[\cite{MT}]\label{teoonerel}
  Every one-relator group with non-trivial center is LERF.
\end{teo}

It is also a well-known fact that every finitely generated abelian group is LERF.

Although it is not true in general that subgroup separability is preserved by direct products, it happens to be true if we add an extra hypothesis. We say a group $G$ is {\bfseries Extended Residually Finite} or just {\bfseries ERF} if every (not necessarily finitely generated) subgroup of $G$ is closed in the profinite topology of $G$.

\begin{teo}[\cite{AG}, Th. 4]\label{teoerf}
  Let $N$ be an ERF group and $M$ be a LERF group. Then any semidirect product $N\rtimes M$ is LERF.
\end{teo}

\begin{lema}\label{lemalerfzn}
  Let $G$ be a group. If $G$ is LERF then $G\times \BZ$ is LERF.
\end{lema}

\begin{proof}
  Since $\BZ$ is LERF (since it is abelian) and every subgroup of $\BZ$ is finitely generated then $\BZ$ is ERF. So the result follows from Theorem \ref{teoerf}.
\end{proof}

\section{Proof of new results}

From now on, $A(\Gamma)$ is always an Artin group as explained in section 2.

To prove our results, we need the Generalized Tits Conjecture for some specific Artin groups.

\begin{lema}\label{lemapre3v}
The Generalized Tits Conjecture (short version) holds for $A=A(\Gamma)$ if $\Gamma$ is connected and has exactly three vertices.
\end{lema}

\begin{proof}
  There are two possibilities for $\Gamma$: it may be a path of length two or a triangle.

  If the first possibility is true, then $A$ is locally reducible since $\Gamma$ contains no triangles.

  Suppose $\Gamma$ is a triangle and $A$ is not locally reducible. Then the labels of the edges of $\Gamma$ need to be $(2,3,3)$, $(2,3,4)$ or $(2,3,5)$. Those are exactly the labels of the irreducible spherical types $A_3$, $B_3$ and $H_3$, respectively.

  So the result follows from Theorem \ref{teojs}.
\end{proof}

As a consequence of Lemma \ref{lemapre3v}, we prove a sufficient condition for $A$ to be not LERF.

\begin{lema}\label{lema3v}
If $\Gamma$ contains a full connected subgraph $\gamma$ of exactly three vertices with at most one of the edges labeled by 2 then $A=A(\Gamma)$ is not LERF.
\end{lema}

\begin{proof}

  Let $B:=A(\gamma)$, isomorphic to a subgroup of $A$. By Lemma \ref{lemapre3v}, $B$ satisfies the Generalized Tits Conjecture hence $$R(B)\leq B\leq A.$$

There are various possibilities for $\gamma$ and we will see that in all of them $R(B)$ is not LERF. That implies $A$ is not LERF, by Theorem \ref{teosublerf}.

Suppose $\gamma$ is a triangle. Then we have two possibilities:

  \begin{enumerate}
  \item
  {\bfseries $\gamma$ has no edges labeled by 2.}

  In that case, $\gamma$ may be written as

  
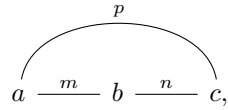
\begin{equation*}\label{eqred4}
\begin{tikzcd}
a  \ar[dash, bend left=80]{rr}{p}  \ar[dash]{r}{m} & b \ar[dash]{r}{n} & c,
\end{tikzcd}
\end{equation*}

with $m,n,p>2$. So
$$\mathcal{S}(B)=\{a, b, c, x=\{a,b\}, y=\{b,c\}, z=\{a,c\}\}.$$

Note that $B$ is not irreducible spherical, since $\gamma$ has no edge labeled by 2.

So $R(B)$ is the RAAG with

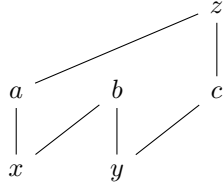
\[
\begin{tikzcd}
&&z \ar[dash]{d} \ar[dash]{dll}\\
a  & b  & c\\
x \ar[dash]{u} \ar[dash]{ur}&  y \ar[dash]{u} \ar[dash]{ur}     &
\end{tikzcd}
\]

as underlying graph. By Theorem \ref{teoraag}, $R(B)$ is not LERF, since its underlying graph contains a full path of length three (\eg $a,x,b,y$).

\item
    {\bfseries $\gamma$ has exactly one edge labeled by 2.}

In that case, $\gamma$ may be written as

    
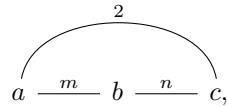
\begin{equation*}\label{eqred2}
\begin{tikzcd}
a \ar[dash, bend left=80]{rr}{2} \ar[dash]{r}{m} & b \ar[dash]{r}{n} & c,
\end{tikzcd}
\end{equation*}

with $m,n>2$. So
$$\mathcal{S}(B)\supset\{a, b, c, x=\{a,b\}, y=\{b, c\}\}=:T.$$

We may not have equality on the above equation because if $B$ is irreducible spherical then $\mathcal{S}(B)$ contains another vertex $z=\{a,b,c\}$. Either way the full subgraph (of the underlying graph of $R(B)$) generated by $T$ is

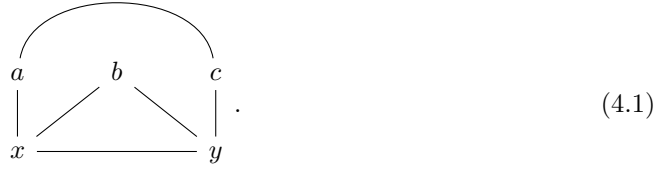
\begin{equation}\label{eqredgrafo}
\begin{tikzcd}
a \ar[dash, bend left=80]{rr} & b & c\\
   x  \ar[dash]{u} \ar[dash]{ur} \ar[dash]{rr}  &  & y \ar[dash]{ul} \ar[dash]{u}
\end{tikzcd}.
\end{equation}

Let $C$ be the RAAG with \eqref{eqredgrafo} as underlying graph. Then $C\leq R(B)$.

By Theorem \ref{teoraag}, $C$ is not LERF, since it contains a full path of length three (\eg $b,x,a,c$). Then $R(B)$ is not LERF, by Theorem \ref{teosublerf}.

\end{enumerate}

Now suppose $\gamma$ is not a triangle. Then $B$ is not irreducible spherical, since $\gamma$ is not a complete graph. We have two possibilities for $\gamma$:

\begin{enumerate}
  \item
  {\bfseries $\gamma$ has no edges labeled by 2}

In that case, $\gamma$ may be written as

 \begin{equation*}
\begin{tikzcd}
a  \ar[dash]{r}{m} & b \ar[dash]{r}{n} & c,
\end{tikzcd}
\end{equation*}

with $m,n>2$. So
$$\mathcal{S}(B)=\{a, b, c, x=\{a,b\}, y=\{b,c\}\}.$$

Then $R(B)$ is the RAAG with

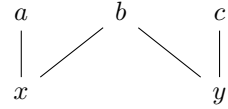
\[
\begin{tikzcd}
a  & b  & c\\
x \ar[dash]{u} \ar[dash]{ur}&       & y \ar[dash]{u} \ar[dash]{ul}
\end{tikzcd}
\]

as underlying graph. By Theorem \ref{teoraag}, $R(B)$ is not LERF, because its underlying graph contains a full path of length three (\eg $a,x,b,y$).

\item
{\bfseries $\gamma$ has exactly one edge labeled by 2.}

In that case, $\gamma$ may be written as

 \begin{equation*}\label{eqred3}
\begin{tikzcd}
a  \ar[dash]{r}{m} & b \ar[dash]{r}{2} & c,
\end{tikzcd}
\end{equation*}

with $m>2$. So

$$\mathcal{S}(B)=\{a, b, c, x=\{a,b\}\}.$$

Then $R(B)$ is the RAAG with

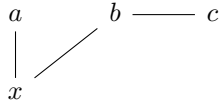
\[
\begin{tikzcd}
a  & b \ar[dash]{r} & c\\
x \ar[dash]{u}  \ar[dash]{ur}&       &
\end{tikzcd}
\]

as underlying graph. By Theorem \ref{teoraag}, $R(B)$ is not LERF, since its underlying graph contains (is actually equal to) a full path of length 3.
\end{enumerate}
\end{proof}

The next two lemmas deal with some important specific examples.

\begin{lema}\label{lemauvw}
  Let $A=A(\Gamma)$ be a LERF Artin group and let $u,v\in V(\Gamma)$. If $v$ is $m-adjacent$ to $u$, with $m>2$, then every other vertex that is adjacent to $u$ is 2-adjacent to $u$ and $v$.
\end{lema}

\begin{proof}
  Let $\gamma$ be the full subgraph of $\Gamma$ generated by $\{u,v,w\}$, with $w$ adjacent to $u$.  Then $A(\gamma)$ is LERF by Theorem \ref{teosublerf}. By Lemma \ref{lema3v}, $\gamma$ has at least two edges labeled by 2, so the result follows.
\end{proof}

\begin{lema}\label{lemasquare}
 Let $A=A(\Gamma)$ be a LERF Artin group. Let $v,u\in V(\Gamma)$ be such that $v$ and $u$ are $m-adjacent$, with $m>2$. Then every other two vertices of $\Gamma$ that are adjacent to $u$ are 2-adjacent to each other.
\end{lema}
\begin{proof}
   Suppose $x,y\in V(\Gamma)\setminus \{v,u\}$ such that $x\neq y$ are adjacent to $u$ and $x$ is not 2-adjacent to $y$. By Lemma \ref{lemauvw}, $x$ and $y$ are 2-adjacent to $u$ and $v$. Let $\gamma$ be the full subgraph generated by $\{x,y,u,v\}$, which may be written as
   
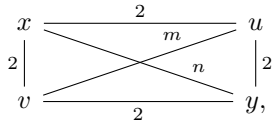
\[
\begin{tikzcd}
x  \ar[dash]{rrr}{2} \ar[dash,near end]{drrr}{n} &&& u \ar[dash]{d}{2}  \\
v \ar[dash]{u}{2} \ar[dash, near end]{urrr}{m}  &&&  y, \ar[dash]{lll}{2}
 \end{tikzcd}
\]

with $n>2$ (consider $n=\infty$ if $x$ and $y$ are not adjacent). By Theorem \ref{teocp}, $A(\gamma)$ contains a right-angled Artin subgroup $B$ whose underlying graph is a square. The group $B$ is not LERF by Theorem \ref{teoraag}, hence $A$ is not LERF by Theorem \ref{teosublerf}, which is a contradiction.
\end{proof}

Now we need a couple of definitions to ease notation.

Let $\Gamma'$ be a graph. Then we will denote the 2-cone of $\Gamma'$ with apex $u$ (as defined in the introduction) as $\Gamma' + \{u\}$.  If $\Gamma$ is a graph, we also define
$$Z(\Gamma):=\{u\in V(\Gamma)\tq \text{$u$ is 2-adjacent to every $v\in V(\Gamma)$ such that $v\neq u$}\}.$$
Note that if $u\in Z(\Gamma)$ then $\Gamma=\Gamma' + \{u\}$, where $\Gamma'$ is the full subgraph generated by $V(\Gamma)\setminus \{u\}$.

We will also need the following easy facts, that follow from the definition of Artin groups: If $\Gamma'$, $\Gamma_1$ and $\Gamma_2$ are graphs then $A(\Gamma'+\{u\})\simeq A(\Gamma')\times \BZ$ and $A(\Gamma_1\sqcup \Gamma_2)\simeq A(\Gamma_1)\ast A(\Gamma_2)$, where $\Gamma_1 \sqcup \Gamma_2$ is the disjoint union of the graphs.

Most of the work of proving Theorem A is dealing with the connected case, which we do in the next two lemmas.

\begin{lema}\label{lemaconcase}
  Let $A=A(\Gamma)$ be a LERF Artin group such that $\Gamma$ is connected. Let $v\in V(\Gamma)$  and let $\Gamma_0$ be the full subgraph of $\Gamma$ generated by $V(\Gamma)\setminus \{v\}$.  If $\Gamma_0$ is disconnected then $v\in Z(\Gamma)$.
\end{lema}

\begin{proof}
  Let $\Gamma_0=\Gamma_1\sqcup \Gamma_2\sqcup \cdots \sqcup \Gamma_n$ be the disjoint union of the connected components of $\Gamma_0$, with $n\geq 2$. Since $\Gamma$ is connected, then $v$ is adjacent to some vertex $u_i \in V(\Gamma_i)$ for each $i$.

   Note that if $u\in V(\Gamma)$ is adjacent to $v$ then $v$ is adjacent to some vertex of $\Gamma$ that is not adjacent to $u$ hence Lemma \ref{lema3v} implies $u$ is 2-adjacent to $v$. That means every vertex of $\Gamma$ that is adjacent to $v$ is actually 2-adjacent to $v$.

  Now let $u\in V(\Gamma_i)$ be a vertex that is not adjacent to $v$. Since $\Gamma_i$ is connected then there is a path $\gamma$ that connects $u$ and $u_i$. Let $u'$ be the vertex of $\gamma$ as close to $u_i$ as possible such that $u'$ is not adjacent to $v$. That means there is a vertex $u_i'\in V(\gamma)$ that is adjacent to $u'$ and 2-adjacent to $v$. By Lemma \ref{lema3v}, $u'$ is 2-adjacent to $u_i'$. Let $\gamma'$ be the full subgraph generated by $\{u', u_i',  v, u_j\}$, for any $j\neq i$. Then $\gamma'$ is a full path of $\Gamma$ of length three with edges labeled by 2. The group $A(\gamma')$ is a subgroup of $A(\Gamma)$, but $A(\gamma')$ is not LERF by Theorem \ref{teoraag}. That contradicts Theorem \ref{teosublerf}.

  So every vertex of $\Gamma_0$ is 2-adjacent to $v$ hence the result holds.
\end{proof}

\begin{lema}\label{lemacongamma}
 Let $A=A(\Gamma)$ be a LERF Artin group such that $\Gamma$ is connected and has at least three vertices. Then $Z(\Gamma)\neq \emptyset$.
\end{lema}
\begin{proof}

We have two possibilities:

\begin{enumerate}
  \item
  {\bfseries All edges of $\Gamma$ are labeled by 2}

  Note that $A$ is a Right-Angled Artin group. We will prove the result by induction on $|V(\Gamma)|$. The case $|V(\Gamma)|=3$ is obvious. Suppose $|V(\Gamma)|> 3$ and assume the result holds for graphs that have less vertices than $\Gamma$.

  Let $v\in V(\Gamma)$ and let $\Gamma_0$ be the full subgraph generated by $V(\Gamma)\setminus \{v\}$. If $\Gamma_0$ is disconnected the result holds by Lemma \ref{lemaconcase}. If $\Gamma_0$ is connected, by induction hypothesis $\Gamma_0=\Gamma' + \{u\}$ for some $u\in \Gamma_0$. Since $\Gamma$ is connected, $v$ is adjacent to some vertex of $\Gamma$. There are two possibilities:

  \begin{enumerate}
    \item
    {\bfseries $v$ is adjacent to $u$.}

    Then $u\in Z(\Gamma)$ and the result holds.

    \item
    {\bfseries $v$ is not adjacent to $u$ and $v$ is adjacent to some $w\in V(\Gamma')$.}

    If $w$ is adjacent to every other vertex of $\Gamma'$, then $w\in Z(\Gamma)$ and the result holds. If that's not the case, then there is $t\in V(\Gamma')$ such that $w$ is not adjacent to $t$. Then the full subgraph of $\Gamma$ generated by $\{v, w, u, t \}$ is either a full square (if $v$ is adjacent to $t$) or a full path of length three (if $v$ is not adjacent to $t$), which is a contradiction by Theorem \ref{teoraag}.
  \end{enumerate}

\item
{\bfseries $\Gamma$ contains at least one edge labeled by $m>2$.}

We again proceed by induction on $|V(\Gamma)|$. Suppose $|V(\Gamma)|=3$. Then by Lemma \ref{lema3v} we have that $\Gamma$ is a triangle with labels $(m,2,2)$ hence $Z(\Gamma)\neq \emptyset$.

Suppose $|V(\Gamma)|> 3$ and assume the result holds for graphs that have less vertices than $\Gamma$.

Let $v\in V(\Gamma)$ and let $\Gamma_0$ be the full subgraph generated by $V(\Gamma)\setminus \{v\}$.  If $\Gamma_0$ is disconnected the result holds by Lemma \ref{lemaconcase}. If $\Gamma_0$ is connected, by induction hypothesis $\Gamma_0=\Gamma_0' + \{u\}$ for some $u\in V(\Gamma_0)$. There are two possibilities:

\begin{enumerate}
  \item
  {\bfseries $\Gamma_0$ contains an edge labeled by $m>2$.}

  We have three possibilities:

  \begin{enumerate}
    \item
    {\bfseries $v$ is $n$-adjacent to $u$, with $n>2$.}

    Since $u$ is 2-adjacent to every vertex of $V(\Gamma)\setminus \{u,v\}=V(\Gamma_0')$, so is $v$ by Lemma \ref{lemauvw}. Then $\Gamma_0'$ is complete with all edges labeled by 2, by Lemma \ref{lemasquare}. That contradicts hypothesis (a).

    \item
    {\bfseries $v$ is 2-adjacent to $u$.}

    Then $u\in Z(\Gamma)$ and the result holds.

    \item
    {\bfseries $v$ is not adjacent to $u$}

    Let $\Gamma_1$ be the full subgraph of $\Gamma$ generated by $V(\Gamma_0')\sqcup \{v\}=V(\Gamma)\setminus \{u\}$. If $\Gamma_1$ is disconnected the result holds by Lemma \ref{lemaconcase}. If $\Gamma_1$ is connected then $Z(\Gamma_1)\neq \emptyset$ by induction hypothesis. We have two possibilities:

    \begin{enumerate}
      \item
      {\bfseries $Z(\Gamma_1)\neq \{v\}$.}

      Then there is  $w\in Z(\Gamma_1)$ such that $w\neq v$ and $w$ is 2-adjacent to $u$, since $u\in Z(\Gamma_0)$. Then $w\in Z(\Gamma)$ hence the result holds.

      \item
      {\bfseries $Z(\Gamma_1)= \{v\}$.}

      Then $v$ is 2-adjacent to all vertices of $\Gamma$ except for $u$, to which $v$ is not adjacent. By hypothesis $\Gamma$ contains an edge labeled by $m>2$, let's say it connects $x,y\in V(\Gamma)$. Then $v$ is 2-adjacent to $x$ and $y$ and so is $u$. By Lemma \ref{lemasquare}, $v$ is 2-adjacent to $u$, which is a contradiction.

    \end{enumerate}
\end{enumerate}

 \item
  {\bfseries $\Gamma_0$ does not contain an edge labeled by $m>2$.}

Then by hypothesis $v$ is $m$-adjacent to a vertex of $\Gamma_0$. There are two possibilities:

\begin{enumerate}
  \item
  {\bfseries $v$ is $m$-adjacent to $u$.}

By Lemma \ref{lemauvw}, $v$ is 2-adjacent to every vertex of $V(\Gamma)\setminus \{u,v\}$. By Lemma \ref{lemasquare}, the full subgraph of $\Gamma$ generated by  $V(\Gamma)\setminus \{u,v\}$ is complete, with all edges labeled by 2. Since $v$ and $u$ are 2-adjacent to those vertices, then they all belong to $Z(\Gamma)$ hence the result holds.

  \item
  {\bfseries $v$ is $m$-adjacent to some $w\in V(\Gamma_0)$.}

  Since $w$ is 2-adjacent to $u$, then $v$ is 2-adjacent to $u$ by Lemma \ref{lemauvw}. Then $u\in Z(\Gamma)$ hence the result holds.

\end{enumerate}
\end{enumerate}

\end{enumerate}

\end{proof}

\begin{lema}\label{lema2v}
 If $\Gamma$ is connected and has at most two vertices then $A=A(\Gamma)$ is LERF.
\end{lema}

\begin{proof}
If $\Gamma$ has only one vertex, then $A\simeq \BZ$ is abelian hence LERF.

If $\Gamma$ has two vertices, then $A$  is either abelian (if the edge label is equal to 2) or irreducible spherical (type $A_2$, $B_2$ or $I_2(p)$), hence it has infinite cyclic center. Since $A$ is by definition a one-relator group, the result follows from Theorem \ref{teoonerel}.
\end{proof}

We can now proceed to the proof of Theorem A.

\begin{proof}[Proof of Theorem A]

  If $\Gamma$ has at most two vertices then $A$ is LERF by Lemma \ref{lema2v}. If $\Gamma_1$ and  $\Gamma_2$ are disjoint graphs that belong to  $\mathcal{S}$ such that $A(\Gamma_1)$ and $A(\Gamma_2)$ are LERF, then $A(\Gamma_1\sqcup \Gamma_2)\simeq A(\Gamma_1)\ast A(\Gamma_2)$ is LERF by Theorem \ref{teofreeprod}. If $\Gamma\in \mathcal{S}$, $u\notin V(\Gamma)$ and $A(\Gamma)$ is LERF then $A(\Gamma+\{u\})\simeq A(\Gamma)\times \BZ$ is LERF by Lemma \ref{lemalerfzn}. So $\Gamma\in \mathcal{S}$ implies $A(\Gamma)$ is LERF.

  Now suppose $A=A(\Gamma)$ is LERF. We will prove that $\Gamma\in \mathcal{S}$ by induction on $|V(\Gamma)|$. If $|V(\Gamma)|\in \{1,2\}$, the result follows from the definition of $\mathcal{S}$.   Suppose $|V(\Gamma)|\geq 3$ and assume the result holds for graphs with less vertices than $\Gamma$.

  If $\Gamma$ is disconnected, then $\Gamma=\Gamma_1 \sqcup \Gamma_2$, with $\Gamma_1, \Gamma_2$ non-empty. Since $|V(\Gamma_i)|<|(V(\Gamma)|$ for i=1,2 and $A(\Gamma_i)$ is LERF for $i=1,2$ by Theorem \ref{teosublerf}, then $\Gamma_1, \Gamma_2\in \mathcal{S}$ by induction hypothesis. By definition of $\mathcal{S}$, we have $\Gamma=\Gamma_1 \sqcup \Gamma_2\in \mathcal{S}$.

  Now suppose $\Gamma$ is connected. Then by Lemma \ref{lemacongamma} we have $\Gamma=\Gamma' + \{u\}$ for some full non-empty subgraph $\Gamma'$ of $\Gamma$ and some $u\in V(\Gamma)$. The group $A(\Gamma')$ is LERF by Theorem \ref{teosublerf} and since $|V(\Gamma')|<|V(\Gamma)|$ then we have $\Gamma'\in \mathcal{S}$ by induction hypothesis. By definition of $\mathcal{S}$, $\Gamma=\Gamma' + \{u\}\in \mathcal{S}$.
\end{proof}

  \begin{proof}[Proof of Corollary A]
Let $A=A(\Gamma)$ be an Artin group. By Theorem A, $A$ is LERF if and only if $\Gamma\in \mathcal{S}$. So it is enough to prove that $A(\mathcal{S})$ is formed exactly by the Artin groups whose underlying graphs belong to $\mathcal{S}$, which easily follows from the comments before Lemma \ref{lemaconcase}.

\end{proof}

{\bfseries Acknowledgments}: We thank A. Minasyan for pointing out a mistake of ours and talking to us about Lemma \ref{lemalerfzn}. We also thank the referee for the careful reading and suggestions.

\end{document}